\documentclass[12pt,a4paper]{article}
\usepackage{amssymb,amsmath,url,paralist, latexsym}

\def\pg{\mbox{\rm PG}}
\def\pgl{\mbox{\rm PGL}}

\def\SL{\mbox{\rm SL}}
\def\PSL{\mbox{\rm PSL}}
\def\gl{\mbox{\rm GL}}
\def\gf{\mbox{\rm GF}}
\def\q{\mbox{\rm Q}}

\def\w{\mbox{\rm W}}

\def\isdef{\stackrel{\mathrm{def}}=}

\def\A{\mathcal{A}}
\def\C{\mathcal{C}}
\let\oldS\S
\def\S{\mathcal{S}}
\def\P{\mathcal{P}}
\def\B{\mathcal{B}}
\def\I{\mathrel{\mathrm I}}
\def\O{\mathcal{O}}

\def\K{\mathcal{K}}

\def\Q{\mathcal{Q}}
\def\IS{\S=(\P,\B,\I)}
\def\Tr{\mathrm{Tr}\,}

\def\be#1{\begin{equation}\label{#1}}
\def\ee{\end{equation}}
\def\bee#1{\begin{equation}\label{#1}\def\arraystretch{1.2}\begin{array}{lllllllllllll}}
\def\eee{\end{array}\end{equation}}
\def\beee{\[\def\arraystretch{1.2}\begin{array}{lllllllllllll}}
\def\eeee{\end{array}\]}

\renewenvironment{matrix}{\left(\!\begin{array}{cccccccccccc}}{\end{array}\!\right)}
\newenvironment{determ}{\left|\!\begin{array}{cccccccccccc}}{\end{array}\!\right|}

\newenvironment{proof}{\noindent{\bf Proof. }}{\ignorespaces\rule{1pt}{0pt}\hfill $\square$\medskip\smallskip\smallskip}
\newtheorem{lemma}{Lemma}
\newtheorem{theorem}[lemma]{Theorem}

\newtheorem{proposition}[lemma]{Proposition}

\title{The known maximal partial ovoids of size $q^2-1$ of $\q(4,q)$}
\author{Kris Coolsaet, Jan De Beule\thanks{The author is a postdoctoral research fellow of the Research Foundation Flanders -- Belgium (FWO).}, and Alessandro Siciliano}
\date{}

\begin{document}
\maketitle

\begin{abstract}
  We present a description of maximal partial ovoids of size $q^2-1$
  of the parabolic quadric $\q(4,q)$ as sharply transitive subsets of
  $\SL(2,q)$ and show their connection with spread sets. This
  representation leads to an elegant explicit description of all known
  examples. We also give an alternative representation of these
  examples which is related to root systems.
\end{abstract}


\section{Introduction}

A (finite) \emph{generalized quadrangle} (GQ) is an incidence structure
$\IS$ in which $\P$ and $\B$ are disjoint non-empty sets of objects called
points and lines (respectively), and for which $\I\ \subseteq (\P \times
\B) \cup (\B \times \P)$ is a symmetric point-line incidence relation
satisfying the following axioms:
\begin{itemize}
\item[(i)] each point is incident with $1+t$ lines $(t \geqslant 1)$ and
two distinct points are incident with at most one line;
\item[(ii)] each line is incident with $1+s$ points $(s \geqslant 1)$ and
two distinct lines are incident with at most one point;
\item[(iii)] if $x$ is a point and $L$ is a line not incident with $x$,
then there is a unique pair $(y,M) \in \P \times \B$ for which $x \I M \I
y \I L$.
\end{itemize}
The integers $s$ and $t$ are the parameters of the GQ and $\S$ is said to
have order $(s,t)$. If $s=t$, then $\S$ is said to have order $s$. If $\S$
has order $(s,t)$, then $|\P| = (s+1)(st+1)$ and $|\B| =
(t+1)(st+1)$ (see e.g. \cite{PayneThas1984}). 

If we interchange the roles of points and lines in a GQ we obtain a
new GQ $(\B,\P,\I)$ which is called the \emph{dual} of the original.

An {\em ovoid} of a GQ $\S$ is a set $\O$ of points of $\S$ such that every line
is incident with exactly one point of the ovoid. An ovoid of a GQ of order
$(s,t)$ has necessarily size $1+st$. A {\em partial ovoid} of a
GQ is a set $\K$ of points such that every line contains {\em at most} one point
of $\K$. The difference $\rho = st + 1- |\K|$ between the size of an
ovoid and the size of a particular partial ovoid $\K$, is called the
{\em deficiency} of $\K$. (Hence $\rho = 0$ if and only if $\K$ is an ovoid.)

A partial ovoid $\K$ is called {\em maximal} if and only if $\K \cup
\{p\}$ is not a partial ovoid for any point $p \in \P \setminus \K$,
in other words, if $\K$ cannot be extended to a larger partial ovoid.

It is a natural question to study {\em extendability} of partial
ovoids, i.e., for what values of $\rho$ can a partial ovoid of
deficiency $\rho$ be guaranteed to extend to a full ovoid?  The
following theorem is a typical result in this context.
\def\citeihatelatex#1#2{\cite[#1]{#2}}
\begin{theorem}[\citeihatelatex{2.7.1}{PayneThas1984}]\label{theo1}
  Let $\IS$ be a GQ of order $(s,t)$. Any partial ovoid of size
  $st-\rho$, $0 \leq \rho < \frac{t}{s}$ is contained in a uniquely
  defined ovoid of $\S$.
\end{theorem}

Remark that if no ovoids of a particular GQ exist, then
Theorem~\ref{theo1} implies an upper bound on the size of partial
ovoids. The following theorem deals with the limit situation, and will
be of use in Section~\ref{sec:geom}.

\begin{theorem}[\citeihatelatex{2.7.2}{PayneThas1984}]\label{theo2}
Let $\IS$ be a GQ of order $(s,t)$. Let $\K$ be a maximal partial ovoid of size
$st-t/s$ of $\S$. Let $\B'$ be the set of lines incident with no point
of $\K$, and let $\P'$ be the set of points on at least one line of $\B'$ and
let $\I'$ be the restriction of $\I$ to points of $\P'$ and lines of $\B'$. Then
$\S'=(\P',\B',\I')$ is a subquadrangle of order $(s,t/s)$.
\end{theorem}

Consider the parabolic quadric $\q(4,q)$ in the 4-dimensional
projective space $\pg(4,q)$.  This quadric consists of points of
$\pg(4,q)$ that are singular with respect to a non-degenerate
quadratic form on $\pg(4,q)$, which is, up to a coordinate transformation,
unique. 

The points and totally isotropic lines of $\q(4,q)$ constitute
an example of a generalized quadrangle of order $q$.
Any elliptic quadric contained in $\q(4,q)$, obtained from a hyperplane section, is an example of 
an ovoid of $\q(4,q)$. These are the only ovoids when $q$ is a prime, \cite{BallGovaertsStorme}.
When $q=p^h$, $h > 1$, other examples are known, see e.g. \cite{DeBeuleKleinMetsch11} 
for a list of references. 

Applying Theorem~\ref{theo1} to the GQ $\q(4,q)$ implies that a partial ovoid
of size $q^2$ cannot be maximal. In this paper we shall be concerned with the next
case, that of maximal partial ovoids of size $q^2-1$. It is 
shown in \cite{DeBeuleGacs} that maximal partial
ovoids of $\q(4,q)$ of size $q^2-1$ do not exist when $q$ is odd and
not prime. When $q$ is odd and prime, examples of maximal partial
ovoids of this size are known for $q=3,5,7$ and $11$, but none for $q
> 11$, \cite{Penttila}. In this paper we will give detailed descriptions of 
exactly these examples. It was also shown in \cite{DWT2011} that $q=3,5,7$ and $11$ 
are the only values permitted under the additional assumption 
that $(q^2-1)^2$ divides the automorphism group of the maximal partial ovoid.

For the sake of completeness, we mention that for $q$ even and $q > 2$, maximal partial ovoids of size $q^2-1$
are excluded, and more is known, by the following theorem. 
\begin{theorem}[{\cite[Corollary 1]{BDES}}]
Let $\K$ be a maximal partial ovoid of $\q(4,q)$, $q$ even. Then $|\K| \leq q^2-q+1$.
\end{theorem}

For the case $q=2$ it is easily seen that here exist maximal partial ovoids of size $q^2-1=3$, see e.g. \cite{Thas2002}.

In Section~\ref{sec:geom} we shall restrict ourselves to $q$ odd, and show that maximal partial ovoids of
size $q^2-1$ can be represented as sharply transitive subsets of the
special linear group $\SL(2,q)$ and show how this representation
naturally leads to a uniform description of the known examples (cf.\
Theorem~\ref{theo-subgroup}). In Section~\ref{sec:spreads} we
illustrate the connection between these partial ovoids and spread
sets, and hence with partial spreads of the symplectic geometry
$\w(3,q)$. Finally, in Section~\ref{sec:rootsys} we present another
(and quite different)
way to construct the known examples, in terms of root systems.

\section{The geometry of Q(4,q) and SL(2,q)}
\label{sec:geom}

{}From now on we shall assume that $q$ is odd.

If a maximal partial ovoid $\O$ of
$\q(4,q)$ has size $q^2-1$, it follows from Theorem~\ref{theo2} that the lines of $\q(4,q)$ not 
meeting $\O$ constitute a subGQ of order $(q,1)$, which is necessarily a hyperbolic quadric $\q^
+(3,q)$ contained in the intersection of a hyperplane $\pi_{\infty}$ and $\q(4,q)$.

It is therefore natural to consider the geometry $\q^*(4,q)$ that
consists of the points of $\q(4,q)$ that do not belong to $\pi_\infty$ (the `affine
points') together with the lines of $\q(4,q)$ that do not lie entirely
inside $\pi_\infty$ and therefore intersect $\pi_\infty$ in exactly
one point (the `affine lines'). There are $q(q^2-1)$ affine
points and $(q+1)(q^2-1)$ affine lines. Each affine line contains $q$ affine points and each affine
point lies on $q+1$ affine lines.

A maximal partial ovoid $\O$ of $\q(4,q)$ of size $q^2-1$ is then
precisely a set of affine points such that each affine line contains
exactly one point of $\O$. Such a set may as well be dubbed an `affine
ovoid'.

Without loss of generality we may choose the equation of $\q(4,q)$ to
be $X_0^2=X_1X_4-X_2X_3$ and the equation of $\pi_\infty$ to be
$X_0=0$. 
With this
notation, the affine points can all be given normalized coordinates
with $X_0=1$, and hence are in one--one correspondence with the
quadruples $(X_1,X_2,X_3,X_4)$ such that $X_1X_4-X_2X_3 = 1$. In other
words, the points of $\q^*(4,q)$ are in one--one correspondence with
the $2\times 2$ matrices of determinant one, i.e., with the elements
of $\SL(2,q)$. It is this correspondence and the group structure of
$\SL(2,q)$ which can be exploited to better understand the `affine
ovoids'. 

We are certainly not the first to identify the points of $\q^*(4,q)$
with group elements of $\SL(2,q)$. Indeed, this representation is
related to the much more general interpretation of so-called
span-symmetric GQs as group coset geometries \cite[Theorem
10.7.8]{PayneThas1984}. Fortunately, for the particular case of
$\q^*(4,q)$ we can derive the necessary properties in a much simpler
setting.

The following lemma shows that there is a wide variety of
ways to express collinearity in $\q^*(4,q)$. We write $I$ for the
$2\times 2$ identity matrix.

\begin{lemma}
\label{lemma-sl2q}
  Let $q=p^h$, $p$ prime, $p$ odd. Let $X,Y\in\SL(2,q)$ such that
  $X\ne Y$. Then the following are equivalent
\begin{compactenum}
\item[(i)] $X$ and $Y$ are collinear in $\q^*(4,q)$,
\item[(ii)] $(1-k) X+k Y\in \SL(2,q)$ for all $k \in \gf(q)$,
\item[(iii)] $(1-k) X+k Y\in \SL(2,q)$ for at least one $k \in \gf(q)-\{0,1\}$,
\item[(iv)] $Y-X$ is singular, i.e., $\det(Y-X)=0$,
\item[(v)] $\Tr XY^{-1} = \Tr Y^{-1}X = 2$,
\item[(vi)]$\Tr YX^{-1} = \Tr X^{-1}Y = 2$, 
\item[(vii)] $XY^{-1}$ (and hence $YX^{-1}$) has multiplicative order $p$,
\item[(viii)] $Y^{-1}X$ (and hence $X^{-1}Y$) has multiplicative order $p$.
\end{compactenum}
\end{lemma}
\begin{proof}
Write 
$X=\begin{matrix}X_1&X_2\\X_3&X_4\end{matrix}$.
Because $\det X = 1$,
the inverse $X^{-1}$ of $X$ is the same as the adjoint $X^\#$, hence
$X^{-1}=X^\#=\begin{matrix}X_4&-X_2\\-X_3&X_1\end{matrix}$. 

Note that the adjoint operator is linear on $2\times 2$ matrices. Also
the trace of the adjoint of a matrix is the same as the trace of the
original. As a consequence $\Tr M = \Tr M^{-1}$ for all $M\in\SL(2,q)$,
and in particular $\Tr XY^{-1} = \Tr YX^{-1}$ and this again is equal
to $\Tr Y^{-1}X = \Tr X^{-1}Y$.

Consider $Z=(1-k)X + kY$ for some $k\in\gf(q)$. We have $Z^\# = 
(1-k)X^\# + kY^\# = (1-k)X^{-1} + kY^{-1}$ and therefore
\begin{eqnarray*}
(\det Z) I = ZZ^\# &=& [(1-k)X + kY][(1-k)X^{-1} + kY^{-1}]\\
&=&(1-k)^2 I + k(1-k)XY^{-1} + k(1-k)YX^{-1} + k^2 I \\
&=& (1-2k+2k^2)I + k(1-k)(XY^{-1}+YX^{-1}).
\end{eqnarray*}
Taking the trace of both sides of this equation and dividing by 2, yields
$\det Z = 1-2k+2k^2 + k(1-k)\Tr XY^{-1}$. Hence $\det Z = 1$ if and
only if $k(1-k) (\Tr XY^{-1} - 2) = 0$.

The matrix $Y-X$ is singular if and only if $YX^{-1}-1$ is
singular. Write $U = YX^{-1}$, $t = \Tr U$. Note that $\det U=1$. $U$
satisfies its own characteristic equation, and therefore $U^2-tU+1 =
0$. If $t=2$, this means $(U-1)^2 = 0$ and hence $U-1$ is singular.
Conversely, $U-1$ is singular if and only if $1$ is an eigenvalue of
$U$. Because $\det U=1$, this means that also the other eigenvalue of
$U$ is $1$ and hence $\Tr U$ equals the sum of the eigenvalues, which
is $2$.

Finally, for any $M\in\SL(2,q)$ we have $(M-I)^p = M^p-I$. Hence, if
$M^p=I$, then the minimal polynomial of $M$ must divide $(x-1)^p$ and
be either $x-1$, in which case $M=I$, or $(x-1)^2 = x^2 -2x + 1$, and
then $\Tr M = 2$. Conversely, if $\Tr M=2$, then $(M-I)^2=0$ and hence
also $(M-I)^p=0$.
\end{proof}

The interpretation of the combinatorial problem of partial ovoids as
subsets of group elements immediately provides us with natural
examples of affine ovoids in the following theorem.
\begin{theorem}
\label{theo-subgroup}
Let $G$ denote a subgroup of $\SL(2,q)$ of order $q^2-1$. Then $G$ is
an affine ovoid of $\q^*(4,q)$.
\end{theorem}
\begin{proof}
  Let $X,Y\in G$, $X\ne Y$. Because $p$ does not divide the order of
  $G$, no element of $G$ can have order $p$. In particular $XY^{-1}$
  cannot have order $p$ and therefore by Lemma~\ref{lemma-sl2q} $X$
  and $Y$ cannot be collinear.
\end{proof}

The subgroup structure of $\SL(2,q)$ is well known \cite[Chapter 3,\oldS 6]{Suzuki}. When
$q$ is odd, $\SL(2,q)$ contains a subgroup of order
$q^2-1$ if and only if $q$=3, 5, 7 or 11, as listed in the following
table
\begin{center}
\begin{tabular}{c|c}
Group & Subgroup \\
\hline
SL(2,3) & $Q_8$\\
SL(2,5) & SL(2,3) = $2\cdot A_4$\\
SL(2,7) & GL(2,3) = $2\cdot S_4$\\
SL(2,11)& SL(2,5)
\end{tabular}
\end{center}
As mentioned before, these are the only known examples of affine
ovoids to date (up to
equivalence, cf.\ below).

We conjecture that also the converse of Theorem~\ref{theo-subgroup} is
true --- that every affine ovoid must be (equivalent to) a subgroup of
$\SL(2,q)$, and hence that all affine ovoids are already known.

Two subsets of $\q(4,q)$ are called equivalent if there exists an
automorphism of the generalized quadrangle $\q(4,q)$ that maps the one
set to the other. The following lemma describes some of these
automorphisms that also leave the hyperplane $\pi_\infty$ invariant.

\begin{lemma}
\label{lemma-auto}
Consider $M,N\in \gl(2,q)$ such that $\det M = \det N$. Let $\sigma$ be a
field automorphism of $\gf(q)$. Then the following maps are
automorphisms of the geometry $\q^*(4,q)$.
\begin{equation}
\label{eq-auto}
X \mapsto MX^\sigma N^{-1},\quad
X \mapsto M(X^{-1})^\sigma N^{-1}
\end{equation}
\end{lemma}
\begin{proof}
For each of the maps in (\ref{eq-auto}) the image of
a matrix $X$ of determinant 1 again has determinant 1. Each map therefore preserves
 the point set of $Q^*(4,q)$. Also, $\Tr XY^{-1}$ is
mapped to either $(\Tr XY^{-1})^\sigma$ or $(\Tr YX^{-1})^\sigma$, and
hence by Lemma~\ref{lemma-sl2q} ({\it v--vi}), collinearity is
preserved.
\end{proof}

In particular, multiplication on the left or right by a fixed element
of $\SL(2,q)$ is an automorphism of the geometry. And hence, from
Theorem~\ref{theo-subgroup} it follows that not only every subgroup
$G$ of the appropriate size, but also every coset of that group, is an affine
ovoid of $\q^*(4,q)$, be it equivalent to $G$.

If we are only interested in point subsets $\O$ up to equivalence,
then it follows from Lemma~\ref{lemma-auto} that we may as well
assume that $I\in\O$. Moreover, if $X \in \O \setminus \{I,-I\}$ has $\Tr X =
t$, then we may always find an automorphism of type (\ref{eq-auto}) with
$M=N$ that leaves $I$ invariant, and maps $X$ to a chosen matrix of
trace $t$ (and determinant 1) different from $\pm I$.

This technique can be used to describe the (affine) lines in an
elegant way.
\begin{lemma}
The lines of $\q^*(4,q)$ are precisely the cosets of the Sylow
$p$-subgroups of $\SL(2,q)$. 
\end{lemma}
\begin{proof}
By the above, it is sufficient to prove that a line that contains $I$
is a Sylow $p$-subgroup of $\SL(2,q)$. Moreover, without loss of
generality we may assume that the line contains the matrix $M(1)
\isdef \begin{matrix}1&1\\0&1 \end{matrix}$. It follows that the
points of the line are of the form $M(k) \isdef (1-k) I + k
M(1)$ with $k\in\gf(q)$. We have  $M(k) =
\begin{matrix}1&k\\0&1 \end{matrix}$ and the set $\{ M(k) | k \in
\gf(q)\}$ is a Sylow $p$-subgroup of $\SL(2,q)$, isomorphic to the
additive group of the field $\gf(q)$.
\end{proof}

The next result shows that affine ovoids are connected to so-called transitive subsets of
$\SL(2,q)$. A subset $S$ of $\SL(2,q)$ is called \emph{transitive} if
and only if for any two non-zero vectors $u,v\in \gf(q)^2$ there exist
an element $X\in S$ such that $uX = v$. $S$ is \emph{sharply transitive}
if and only if the element $X$ is always unique.

\begin{theorem}\label{th:sharply}
Let $q$ be odd. Let $\O\subseteq \SL(2,q)$. Then $\O$ is an affine
ovoid of $\q^*(4,q)$ if and only if $\O$ is a sharply transitive
subset of $\SL(2,q)$.
\end{theorem}
\begin{proof}
Let $\O$ be sharply transitive and take $X,Y\in\O$, $X\ne Y$. Then,
for any $u\in\gf(q)^2$, $uX$ and $uY$ must differ, and hence
$u(X-Y)\ne 0$, for all $u\ne (0,0)$. It follows that $X-Y$ is non-singular, and hence $X,Y$
are never collinear, by Lemma~\ref{lemma-sl2q}, making $\O$ an affine
ovoid.

Conversely, let $\O$ be an affine ovoid. For every $X,Y\in \O$, $X-Y$
is non-singular, and hence the set $V = \{uX | X\in\O\}$ has size
$|\O|$. Because $|\O|\ge q^2-1$ the set $V$ must contain every non-zero
vector of $\gf(q)^2$. Therefore $\O$ is transitive, and because $|\O|$
is exactly $q^2-1$, it is even sharply transitive.
\end{proof}

For $t\in \gf(q)$ we define the \emph{discriminant} $\delta(t)$, as follows:
\[
\delta (t) = \left\{\begin{array}{ll}
-1,\quad & \mbox{when $t^2-4$ is not a square in $\gf(q)$,}\\
0,& \mbox{when $t^2- 4 = 0$,}\\
1,& \mbox{when $t^2- 4$ is a non-zero square in $\gf(q)$.}
\end{array}\right.
\]
The quadratic equation $\lambda^2 -t \lambda + 1 = 0$ has exactly $1 + \delta(t)$
solutions for $\lambda\in \gf(q)$. (This is the characteristic
equation of a matrix $X\in\SL(2,q)$ with $\Tr X = t$.)

\begin{proposition}
\label{prop-t}
Let $t\in \gf(q)$. Let $S_t$ denote the set  of all
elements $X$ of $\SL(2,q)$ such that $\Tr X = t$. Then
$S_t = H_t \setminus \pi_\infty$ where $H_t$ is a hyperplane section
of $\q(4,q)$ whose type depends on $\delta(t)$ as follows:
\begin{compactenum}[\rm (i)]
\item If $\delta(t)=-1$, then $H_t$ is an elliptic
  quadric of type $Q^-(3,q)$.
\item If $\delta(t) = 1$, then $H_t$ is a hyperbolic
  quadric of type $Q^+(3,q)$. 
\item If $\delta(t) = 0$, then $H_t$ is a quadratic cone.
\end{compactenum}
In all cases $|S_t|=q(q+\delta(t))$. 
\end{proposition}

\begin{proof}
Let $X=\begin{matrix}
X_1 & X_2 \\ X_3 & X_4 \\
\end{matrix}$ as before. The condition $\Tr X = t$ translates to $X_1
+ X_4 = t$, or in projective coordinates, $X_1 + X_4 = tX_0$  which is
the equation of a hyperplane $H_t$ of $\pg(4,q)$. Such a hyperplane
intersects $\q(4,q)$ in either a cone, an elliptic or a hyperbolic  
quadric. It remains to determine which values of $t$ lead to which type
of intersection.

Combining the equation of $H_t$ with the equation of $\q(4,q)$ yields
\[
X_1(tX_0 -X_1) - X_2X_3 - X_0^2 = 0.
\]
The corresponding quadratic form has the following associated determinant~:
\[
\begin{determ}
-1 & t/2 & 0 & 0 \\
t/2 & -1 & 0 & 0 \\
0 & 0 & 0 & -1/2 \\
0 & 0 & -1/2 & 0
\end{determ}
= (1 - t^2/4)(1/4) = \frac1{16}\delta (t).
\]
The type of the hyperplane section is determined by whether this
determinant is a square, a non-square or zero, yielding the
classification in the statement of this theorem.

It remains to determine the value of $|S_t|$. Note that the size of
the hyperplane section is equal to $q^2+1$, $q^2+q+1$ or $(q+1)^2$
depending on the type of the hyperplane section, i.e., equal to
$q^2+q+1 + \delta(t)q$. From this size we need to subtract the size of
the intersection $S_t\cap \pi_\infty$. This intersection consists of
the points satisfying $X_0=0$ and $-X_1^2 - X_2X_3 = 0$, i.e., a
non-degenerate conic. A conic has $q+1$ points, and therefore $|S_t| = q^2+q+1
+\delta(t) q - (q+1) = q(q+\delta(t))$.
\end{proof}

\begin{lemma}
\label{lemma-count}
Let $\O$ denote an affine ovoid of $\Q^*(4,q)$. Then
\begin{compactenum}[\rm (i)]
\item $\O$ either contains $I$ or else exactly $q+1$ elements of
trace $2$ different from $I$,
\item $\O$ either contains $-I$ or else exactly $q+1$ elements of
trace $-2$ different from $-I$, 
\item $\O$ contains exactly $q+1$ points of trace $t$, for every $t$
such that $\delta(t)>0$.
\end{compactenum}
 Every point of $\Q^*(4,q)$ outside $\O$ 
is collinear with exactly $q+1$ points of $\O$.
\end{lemma}
\begin{proof}
By Proposition~\ref{prop-t} (iii), the points of trace $2$ consist of $q+1$
(affine) lines through the common point $I$. Because $\O$ is an affine ovoid,
each of these lines must contain exactly one point of $\O$. Either
this is the point common to all these lines, i.e., $I$, or else a
different point for each line. This proves the first part of this
lemma. The second part is proved in the same way, by considering the
points of trace $-2$ instead of $2$.

Now, let $t$ be such that $\delta(t) > 0$. By Proposition~\ref{prop-t}
(i) $H_t$ is a hyperboloid and its point can therefore be partitioned
into $q+1$ lines (in two different ways). Each of these lines must
contain exactly one point of $\O$.

Finally, consider a point outside $\O$. Without loss of generality we
may assume this point to be $I$. The first part of this lemma
then proves that this point is collinear to exactly $q+1$ points of
$\O$.
\end{proof}

\begin{lemma}
\label{lemma-antipodal}
Let $A,B\in \SL(2,q)$ such that $A\ne B$. Then $A$ and $B$ are
collinear to the same points of $\pi_\infty$ if and only if $A = -B$.
\end{lemma}
\begin{proof}
Let $A=\begin{matrix}A_1&A_2\\A_3&A_4\end{matrix}$, 
$B=\begin{matrix}B_1&B_2\\B_3&B_4\end{matrix}$. The points 
$(X_0,X_1,\ldots, X_4)$ of
$\pi_\infty$ that are collinear to $A$ (resp.\ $B$) are the points of the
3-dimensional subspace of $\pi_\infty$ with equation $X_0=0$ and $A_2X_1 +
A_1X_2 - A_4X_3 - A_3X_4= 0$ (resp.\ $B_2X_1 +
B_1X_2 - B_4X_3 - B_3X_4= 0$). These two $3$-spaces are the same if
and only if the corresponding quadruples $(A_2,A_1,-A_4,-A_3)$ and 
 $(B_2,B_1,-B_4,-B_3)$ are equal up to a multiplicative factor. In other
 words, if the matrices $A$ and $B$ are equal up to a multiplicative factor. From $\det
 A=\det B = 1$ it follows that this factor can only be $1$ or $-1$.
\end{proof}

Pairs of points $A,-A$ that satisfy the conditions of Lemma
\ref{lemma-antipodal}, shall be called \emph{antipodal}. Note that
antipodality is preserved by the automorphisms of Lemma~\ref{lemma-auto}.

\begin{theorem}\label{theo:pairs}
If the affine ovoid $\O$ is a subgroup of $\SL(2,q)$ then it is the disjoint
union of $\frac12(q^2-1)$ antipodal pairs.
\end{theorem}
\begin{proof}
  Assume the contrary, and let $A\in\O$ such that $-A\notin
  \O$. Without loss of generality we may set $A=I$. By Lemma
  \ref{lemma-count} (ii) $\O$ must contain at least one element
  $X$ with $\Tr X = -2$ but $X\ne -I$. Such $X$ must satisfy its
  characteristic equation $X^2 + 2X + 1 = 0$. Hence $(X+1)^2 = 0$ and
  then $X^p +1= 0$ (see the proof of Lemma~\ref{lemma-sl2q}). It
  follows that $X$ has order $2p$, and hence that $2p$ must divide the
  order $q^2-1$ of the group $\O$. This is a contradiction.
\end{proof}

(In \cite{DWT2011} this theorem was proved for the special case  $q=5$.)

\section{Spreads of W(3,q) and spread sets}
\label{sec:spreads}
In this section we shall describe
an interesting correspondence between the points of
$\q^*(4,q)$ and certain lines of the 3-dimensional projective space
$\pg(3,q)$.

A {\em spread} $\S$ of $\pg(3,q)$ is a set of lines which partition
the point set of $\pg(3,q)$. A spread necessarily contains $q^2+1$
lines.  A {\em partial spread} of $\pg(3,q)$ is a set of mutually non-intersecting lines. 
It follows that a partial spread is a
spread if and only if it has size $q^2+1$.  A partial spread is {\em
  maximal} if it cannot be extended to a larger partial spread.

Every partial spread that is sufficiently large can always be made
into a full spread by adding appropriate lines, as stated in the
following theorem.
\begin{theorem}[{\cite{MR0290242}}]
\label{theo:spread}
Let $\S$ be a partial spread of $\pg(3,q)$ of size $q^2+1-\delta$. If
$\delta \leq \epsilon$, such that $q+\epsilon$ is smaller than the
smallest non-trivial blocking set of $\pg(2,q)$, then $\S$ is
extendable to a spread.
\end{theorem}
(A lower bound for $\epsilon$ is $\sqrt{q}$, and this lower bound is
sharp when $q$ is a square.)

One way to construct spreads is by means of so-called spread sets (see
\cite{MR0233275} for more information).
A \emph{spread set} is a collection $\C$ of $2 \times
  2$-matrices over $\gf(q)$ which satisfies the 
following three conditions:
\goodbreak
\begin{compactenum}[\rm (i)]
\item $|\C|=q^2$;
\item $\C$ contains the zero matrix and the identity matrix;
\item If $X,Y \in \C$, $X \neq Y$, then $\det(X-Y) \neq 0$.
\end{compactenum}

Let $X=\begin{matrix}X_1&X_2\\X_3&X_4\end{matrix}$ denote
a general $2\times 2$ matrix (not necessarily of determinant 1) and
define $L(X)$ to be the line of $\pg(3,q)$ that connects the points
with coordinates $(1,0,X_1,X_2)$ and $(0,1,X_3,X_4)$. 
Recall that the line $L$ through the points $(x_1,x_2,x_3,x_4)$ and
$(y_1,y_2,y_3,y_4)$ can also be represented by its \emph{Pl\"{u}cker
coordinates} $p(L) = (p_{01}, p_{02}, p_{03}, p_{23}, p_{31}, p_{12})$,
satisfying $p_{01}p_{23} + p_{02}p_{31} + p_{03}p_{12} = 0$, where
$p_{ij}\isdef\begin{determ}x_i&x_j\\y_i&y_j
\end{determ}$. The Pl\"{u}cker
  coordinates of $L(X)$ are easily computed to be
\begin{equation}
\label{eq-plucker}
p(L(X)) = (1, X_3, X_4, \det X, -X_1, X_2).
\end{equation}

It is not so difficult to prove that the lines $L(X)$ and $L(Y)$ have
a non-empty intersection
if and only if $\det (X-Y) = 0$.
Hence, if $\C$ is a spread set, then the set of lines
$L(\C) \isdef \{ L(X) \mid X \in \C \}$ is a partial spread of $\pg(3,q)$
of size $q^2$. 

Theorem~\ref{theo:spread} guarantees that we can
find one more line $L'$ such that $L(\C)\cup\{L'\}$ is a full spread
of $\pg(3,q)$, and indeed, in this case it is easily verified that the
line
connecting the points
with coordinates $(0,0,1,0)$ and $(0,0,0,1)$
satisfies this role.
In fact, every spread of $\pg(3,q)$ is equivalent to a
spread which is obtained from a spread set in this way.

Now, affine ovoids of $\SL(2,q)$ yield a natural way to construct
spread sets. Indeed, it follows immediately from Lemma
\ref{lemma-sl2q} that extending an affine ovoid $\O$
with the zero matrix yields a spread set
$\C\isdef\O\cup\{0\}$. The fact that every $X\in\O$ has $\det X=1$
makes this spread set (and the associated spread) rather
special. Indeed, by
(\ref{eq-plucker}), every line of the partial spread $L(\O)$ (i.e., every
line of the full spread, except two) satisfies the identity
$p_{01}=p_{23}$. In other words, every such line is an isotropic line
for the symplectic form
\begin{equation}
\label{eq-form}
\begin{determ}x_0&x_1\\y_0&y_1
\end{determ}
-
\begin{determ}x_2&x_3\\y_2&y_3
\end{determ}.
\end{equation}

If we denote by $\w(3,q)$ the geometry 
where the points are the points of $\pg(3,q)$ and the lines are those
lines of $\pg(3,q)$ that are isotropic with regard to the
symplectic form (\ref{eq-form}), then it follows that 
$L(\O)$ is a partial spread of $\w(3,q)$

It is well known, see e.g. \cite{PayneThas1984}, that $\w(3,q)$ is a
GQ which is equivalent to the dual of $\q(4,q)$. Hence, (maximal)
partial ovoids of $\q(4,q)$ are equivalent to (maximal) 
partial spreads of $\w(3,q)$. The Pl\"{u}cker coordinates in
(\ref{eq-plucker}) demonstrate the explicit correspondence between $L(X)$ and
$X$ for our representation of $\q^*(4,q)$. (This correspondence is
linear because $\det X=1$.)

We could equally well choose to present the theory developed in
Section~\ref{sec:geom} in the framework of partial spreads of
$\w(3,q)$ and spread sets.  The extra condition that $I \in \C$ is not
really a restriction, and is related to the fact that by
Lemma~\ref{lemma-auto} we can also require $I\in\O$ without loss of
generality.

\section{Another explicit description}
\label{sec:rootsys}
In this section we give another explicit description of the
known examples of affine ovoids of $\q^*(4,q)$, although it is not
directly related to $\SL(2,q)$.

We now choose a different
representation of the parabolic quadric $\q(4,q)$, i.e., as the
quadric with equation $X_1^2+X_2^2+X_3^2+X_4^2 = X_0^2$. Two points
$X,Y$ on this quadric are collinear if and only if
$X_1Y_1+X_2Y_2+X_3Y_3+X_4Y_4 = X_0Y_0$.

For $\pi_\infty$ we again take the hyperplane with equation $X_0=0$. 
The affine points of the quadric then satisfy the property $X_0\ne 0$, and
again we may normalize their coordinates by setting $X_0=1$.

We shall consider several sets of vectors of norm 1 in a 4-dimensional real
Euclidean space with the property that the number of mutual inner
products among the vectors is relatively small.
As a first example, consider the following set of 8 vectors~:
\[
\K_8 \isdef \{(\pm 1,0,0,0),(0,\pm 1,0,0), (0,0,\pm 1,0), (0,0,0,\pm 1)\}.
\]
Each vector in this set has norm 1 and inner products between
different elements of $\K_8$ can only take the values $0$ and $-1$, i.e.,
never 1. It follows that the set $\O$ of points with coordinates
$(1,X_1,X_2,X_3,X_4)$, where $(X_1,X_2,X_3,X_4)\in\K_8$, is a set of 8
points of the (real) parabolic quadric where no two points are collinear.
Hence, if we reduce this set modulo 3, we obtain an affine ovoid for
the case $q=3$.

The set $\K_8$ is a root system of rank 4 of type $A_1^4$. Root systems
have the property that they allow only few different values for
inner products. It is therefore natural to investigate whether they
can lead to affine ovoids also for other values of $q$.

Indeed, consider the following root system (of type $D_4$)~:
\[
\K_{24} \isdef \K_8 \cup \{
(\pm \frac{1}{2}, \pm \frac{1}{2}, 
\pm \frac{1}{2}, \pm \frac{1}{2}) \}.
\]
Possible values for inner products of different elements are now $\frac12$,
$0$, $-\frac12$ and $-1$. Hence, reducing modulo 5 yields an affine
ovoid (of size 24) for $q=5$.

The same root system has an alternative representation.
\begin{eqnarray*}
\lefteqn{\K'_{24} \isdef \frac1{\sqrt2} \{ 
(\pm 1, \pm 1, 0, 0), 
(\pm 1, 0, \pm 1, 0),}\\
&&\qquad\qquad\qquad 
(\pm 1, 0, 0, \pm 1), 
(0, \pm 1, \pm 1, 0), 
(0, \pm 1, 0, \pm 1), 
(0, 0, \pm 1, \pm 1) \}.
\end{eqnarray*}
This representation cannot be used in $\gf(5)$ because 2 is not a
square modulo 5. However, in $\gf(7)$ we may write $\sqrt2 = 3$.

It turns out that there are also only a small number of possibilities for
inner products between elements of $\K_{24}$ and $\K'_{24}$, viz.\ $\pm
1/\sqrt2$ and $0$. As a consequence, the set $\K_{48}\isdef
\K_{24}\cup \K'_{24}$ only admits the inner products $-1$, $\pm 1/2$,
$\pm 1/\sqrt2$ and $0$. Modulo 7 all of these are different from 1,
and hence we may use $\K_{48}$ to obtain an affine ovoid when $q=7$.

Note that $\K_{48}$ can be obtained from the root system of type $F_4$ by
normalizing the short and long vectors to become the same length. This
is the largest root system of rank 4, hence for the case $q=11$ we
shall have to look elsewhere.

In that case the 600-cell, a 4-dimensional polytope of type $H_4$,
comes to the rescue. The set $\K_{120}$ of coordinates of the 120
vertices of this polytope can be constructed by extending $\K_{24}$
with the 96 coordinates of the form $\frac12(\pm 1,\pm
\varphi,\pm1/\varphi, 0)$, with $\varphi=\frac12(1+\sqrt5)$ (the
golden ratio), where we allow all \emph{even} permutations of the
coordinates. Inner products among these vertices have values
$-1$, $-\varphi/2$, $-1/2$, $-1/(2\varphi)$, $0$,
  $1/(2\varphi)$, $1/2$, $\varphi/2$ or $1$, where the latter value only
  occurs when both vectors are the same. In $\gf(11)$, $\sqrt 5 = 4$
  and hence $\varphi$ reduces to 8. In other words, $\K_{120}$
  provides an example of an affine ovoid for $q=11$.

Root systems  have been used in the past for constructing other types of combinatorial
object, e.g., (partial) flocks of hyperbolic quadrics by Bader et al.\ \cite{BDLLP2002}.

\section{Concluding remarks}
It was already mentioned that we are convinced that the four affine
ovoids discussed in this paper constitute the full set of existing
examples when $q$ is odd, although so far a proof of this result has completely eluded
us. We think that the representation of affine ovoids within
$\SL(2,q)$ provides the most natural setting for such a proof, as it
allows the use of both group theoretical and combinatorial techniques
for tackling the problem.  On the other hand, the setting of Section
\ref{sec:rootsys} is probably the best way to look for
counterexamples.

By Theorem \ref{th:sharply} our conjecture is equivalent to the
statement that every sharply transitive subset of $\SL(2,q)$ is a
coset of a subgroup. The analogous statement for $\pgl(2,q)$ is true,
cf.\ \cite{DS2003}, and equivalent to the classification of flocks of
the hyperbolic quadric of $\pg(3,q)$, as was first observed by
Bonisoli \cite{Bonisoli1988}. Unfortunately, the techniques used to
prove this statement do not readily carry over to our case.

If our conjecture turns out to be too strong, then at least we expect
all affine ovoids to consist of antipodal pairs (although again, we
had no success in proving this weaker result). In that case, the
problem of classifying all affine ovoids of $\SL(2,q)$ can be reduced
to the same problem for the smaller group $\PSL(2,q)$, an affine ovoid
then being defined as a set of size $\frac12(q^2-1)$ of elements $X,Y$
of $\PSL(2,q)$ such that $\Tr XY^{-1} \ne \pm 2$.

The problem can also be approached through the theory of association
schemes. With $\SL(2,q)$ we may associate a scheme that consists of
the relation of antipodality ($X = -Y$) together with $q$ relations
$R_t$ for $t\in \gf(q)$, where $X\,R_t\ Y$ if and only if $\Tr XY^{-1}
= t$ (and $X\ne \pm Y$). It is not so difficult to compute the various
intersection numbers for this association scheme.
The use of standard methods from the theory of association schemes is
however somewhat hindered by the fact that the number of classes
varies with $q$. 

This setting is closely related to the group representation theory of
$\SL(2,q)$. Indeed, the conjugacy classes of $\SL(2,q)$ correspond to
the sets of matrices of given trace, where extra provision should be
made for the singleton classes $\{I\}$ and $\{-I\}$. In this context
it would also be nice to extend Lemma~\ref{lemma-count} with a
non-trivial result for $\delta(t) < 0$.

Yet another alternative is to employ a computer to at least tackle the
smaller cases. We have proved by computer that for $q=3$, $5$, $7$ and $11$ the affine ovoids are indeed unique and that for $q=9$ none exist
(which confirms the result of \cite{DeBeuleGacs}). For $q=9$ our program takes only a few minutes of CPU time while for $q=11$ already three weeks were needed. We fear that for larger values of $q$ the problem may already be intractable by standard methods. Again a stronger
version of Lemma~\ref{lemma-count} might be of help.

\end{document}